\documentclass[11pt]{amsart}
\usepackage[utf8]{inputenc}
\usepackage[normalem]{ulem}

\usepackage{amsmath, amssymb, amsthm,color, graphicx, enumitem, upgreek} 
\usepackage[unicode]{hyperref}
\usepackage[artemisia]{textgreek}
\usepackage{epsfig,color,mathrsfs,amsbsy,amsfonts,amsmath,amssymb,amsthm,fullpage}

\definecolor{purple}{rgb}{0.59, 0.44, 0.84}

\definecolor{turk}{rgb}{0.2, .60, .50}

\def\e{\mathbf e}
\def\Q{\mathbb Q}
\def\C{\mathbb C}
\def\sT{\mathscr T}

\def\R{\mathbb R}
\def\Z{\mathbb Z}

\def\SL{\mathrm{SL}}

\def \f{\frac}

\def \bk {\color{black}}
\def \bl#1{ {\textcolor{blue} {#1}}   }

\def \ker \text{Ker}

\def \e {{\bf {e}}}
\def \v {{\bf {v}}}
\def \w {{\bf {w}}}
\def \u {{\bf {u}}}

\newtheorem{theorem}{Theorem}
\newtheorem{prop}{Proposition}
\newtheorem{cor}{Corollary}

\newtheorem{lemma}{Lemma}
\newtheorem{Lemma}{Lemma}

\newcounter{example}
\newenvironment{example}[1][]{\refstepcounter{example}\par\medskip
   \noindent \textbf{Example~\theexample. #1} \rmfamily}{\medskip}
\newcounter{remark}

\newcommand{\mr}[1]{\mathrm{#1}}

\newcommand{\ZZ}{\mathbf{Z}}
\newcommand{\FF}{\mathbf{F}}

\newcommand{\CCC}{\mathbf{C}}

\def\F{\FF}
\def\C{\CCC}

\def\Q{\mathbb Q}
\def\R{\mathbb R}

\def\Z{\mathbb Z}

\def \SL{\rm{SL}}

\def \l {\lambda}

\def\({\left(}
\def\){\right)}
\def \f12{\frac12}
\def\G{\Gamma}
\def \bl{\color{blue}}

\def\PSL{\mbox{PSL}}
\def\SL{\mbox{SL}}
\def\SL2Z{\mr{SL}_(\ZZ)}
\def\SL2R{\mr{SL}_2(\mathbf{R} )}
\def\PSL2R{\mr{PSL}_2(\mathbf{R} )}
\def\PSL2Z{\mr{PSL}_2(\ZZ )}

\catcode`,\active

\catcode`\,12

\catcode`,\active

\catcode`\,12

\catcode`,\active

\catcode`\,12

\catcode`,\active

\catcode`\,12

\catcode`,\active

\catcode`\,12

\def \f{\frac}

\def \bk {\color{black}}
\def \bl#1{ {\textcolor{blue} {#1}}   }

\def \ker \text{Ker}

\def \e {{\bf {e}}}
\def \v {{\bf {v}}}
\def \w {{\bf {w}}}
\def \u {{\bf {u}}}

\begin{document}
\title{Hodge numbers of hypergeometric data}

\author{Ling Long}
\address{Department of Mathematics, Louisiana State University, Baton Rouge, LA 70803-4918, USA}
\email{llong@lsu.edu}

\author{Yifan Yang}
\address{Department of Mathematics, National Taiwan University and National Center for Theoretical Sciences, Taipei, Taiwan 10617}
\email{yangyifan@ntu.edu.tw}

\begin{abstract}
    In this paper, based on the toric hypergeometric model given in a paper by Beukers--Cohen--Mellit, we provide two other ways to explain why the zig-zag diagram method can be used to compute Hodge numbers for hypergeometric data defined over $\Q$.
\end{abstract}
\maketitle
\section{Introduction}
Let $HD=\{\alpha=\{a_1,a_2,\cdots,a_m\}, \beta=\{1,b_2,\cdots,b_m\}\}$ with $a_i, b_j \in \Q\cap (0,1]$ be a  hypergeometric datum   defined over $\Q$, which means $\prod_{i=1}^m (X-e^{2\pi i a_i})\in\Z[X]$ and $\prod_{i=1}^m (X-e^{2\pi i b_i})\in\Z[X]$.
The purpose of this note is to compute the Hodge-Tate weights arising from the generic  newton polygon of $HD$  in terms of combinatorics via the Hodge numbers of toric varieties. They form a set of fundamental invariants when we consider the hypergeometric Galois representations arising from $HD$, as in \cite{Allen2024explicit} and \cite{LLT2}. Classically, each Hecke eigenform have two important features: level and weight, say $k$. The latter one means the Hodge-Tate weights of the corresponding Galois representation  are $\{0,k-1\}$. \medskip

Write $$\frac{\prod_{i=1}^m(X-e^{2\pi i a_i})}{\prod_{i=1}^m(X-e^{2\pi i b_i})}=\frac{(X^{p_1}-1)\cdots (X^{p_r}-1)}{(X^{q_1}-1)\cdots (X^{q_s}-1)}$$ where $p_i,q_j\in \mathbb N$, $\{p_i\}\cap\{q_j\}=\emptyset$, and \begin{equation}\label{eq:L}
  \sum_{i=1}^r p_i=\sum_{j=1}^s q_j:=L(HD).  
\end{equation} We call $L(HD)$ the \emph{natural length} of $HD$.
We  also refer to \begin{equation}\label{eq:gamma-v}
   \G(HD):=[p_1,\cdots,p_r,-q_1,\cdots,-q_s] 
\end{equation} as the \emph{hypergeometric gamma vector} of $HD$. Here we assume $q_s=1$. Further we let $n=r+s-1$ and $$M(HD):=\text{lcd}(\alpha\cup \beta)=\text{lcm} (\{p_i,q_j\}),$$ be the \emph{level} of $HD$, where $\text{lcd}$ means least positive common denominator. Let $\bf o$ be the origin of  $\Z^n$ and  fix a set of vectors in $\Z^n$:
\begin{equation}\label{eq:ei}
    \begin{array}{rl}
    \e_0=&[1,0,\cdots,0]\\
    \e_i=&[1,0,\cdots, 0, \underset{i+1}1,0,\cdots,0], \quad 1\le i\le n-1\\
    \e_n=&[1,p_1,\cdots,p_{r-1},-q_1,\cdots,-q_{s-1}]
\end{array}
\end{equation}
Let $\Delta(HD)$ be the convex polyhedron with vertices  $\{{\bf o}\}\cup \{\e_i\}_{i=0}^n$ and $\text{Vol}(\Delta(HD))$ be its volume. Its Hodge numbers  are non-negative integers. 
Let $H_{\Delta(HD)}(k)$ be the multiplicity of $k$ among these Hodge numbers. 
They form a set of useful invariants when we study the toric varieties parameterized by $HD$ as follows.  
Given a gamma vector $\G(HD)$ as above, $\l\in \C$, there is a toric variety given by (1.1) of \cite{BCM} by Beukers, Cohen and Mellit of the form
$$\quad  x_1+\cdots +x_r-y_1-\cdots -y_s=0, \quad \frac{\lambda}C x_1^{p_1}\cdots x_r^{p_r}=y_1^{q_1}\cdots y_s ^{q_s}, \quad C=\frac{\prod_{i=1}^rp_i^{p_i}}{\prod_{i=1}^sq_i^{q_i}}.
$$ 
Following  \cite{BCM}, we consider the inhomogeneous model (note that $q_s=1$) by letting 
$$x_r=1, \quad y_s=\frac\lambda C
\frac{x_1^{p_1}\cdots x_{r-1}^{p_{r-1}}}
{y_1^{q_1}\cdots y_{s-1} ^{q_{s-1}}}
$$
and hence the above variety is reduced to
 \begin{equation}\label{f_HD}f_{HD(\l)}= 1 +x_1+\cdots +x_{r-1}-y_1-\cdots -y_{s-1}- \frac{\l}C \frac{x_1^{p_1}\cdots x_{r-1}^{p_{r-1}}}{y_1^{q_1}\cdots y_{s-1} ^{q_{s-1}}}=0.\end{equation} 

 \begin{example}[{$\G(HD)=[3,-1,-1,-1]$}]\label{eg:1}
In this case, $C=27$, for any $\l\neq 0,1$ $f_{HD}(\l)=\displaystyle 1-y_1-y_2-\frac{\l}{27y_1y_2}=0$ is an elliptic curve. 
\end{example}

 Given a fixed finite field $\F_q$ whose characteristic is assumed to be coprime to $M(HD)$ throughout this paper and let $\l\in \F_q$. To compute the number of solutions of $f_{HD}(\l)=0$  in $\mathbb G_m^{r+s-2}(\F_q):=(\F_q^\times)^{r+s-2}$, one can make use of any nontrivial additive character $\psi$ of $\F_q$ by computing
 \begin{equation}\label{eq:1/q}
     \frac 1q \sum_{\substack{x_0\in\F_q\\ \,\, (x_1,\cdots,y_{s-1}) \in(\F_q^\times)^{r+s-2}}} \psi\left (x_0f_{HD}\right),
 \end{equation}plus contributions from the boundary. 
 See \cite{BCM} for a chosen compactification. Below we regard $x_0f_{HD}(\l)$  as a Laurent polynomial in $n=r+s-1$ variables $x_0,x_1,\cdots, x_{r-1},y_1,\cdots, y_{s-1}$. Belowe we treat it as a vector ${\bold x}$.   Using the trace  map, $\psi$ can be extended to an additive character $\psi_r$ for any degree-$r$ extension $\F_{q^r}$ of $\F_q$. 
 
To continue we first recall some more background information on exponential sums.
Assume that $f(\bold x)$ is a Laurent polynomial in $n$ variables with coefficients in $\F_q$.   Let
 \begin{equation}
     S_r(f;\F_q):=\sum _{\bold x \in \mathbb G_{m}^n (\F_{q^r})}\psi_r(f(\bold x))
 \end{equation}
 Then by Dwork 
 $$Z_q(f;T):=\displaystyle \exp\left((-1)^{n-1}\sum_{r\ge 1} S_r(f;\F_q) \frac{T^r}r\right)
 $$
 is a rational function of $T$. A fundamental result of   Adolphson and Sperber  \cite{Adolphson-Sperber89} says when $f$ is non-degenerate (see Definition 1.1 of \cite{Wan-Newton}),  $Z_q(f;T)$ is a polynomial  in $\Z[T]$ with constant 1, whose degree is $\operatorname{Vol} \Delta_f \cdot n!$ where $\Delta_f$ is a convex polyhedron in $\Q^n$
 whose vertices are $\mathbf{o}$ together with the exponent vectors corresponding to the monomials of $f$ with nonzero coefficients.  
 
Note that for $f=x_0f_{HD}$ the vertex-vectors coincide with $\e_0,\e_1,\cdots, \e_{n}$.
For $\l\neq 0,1$ in $\F_q$, the polynomial $x_0f_{HD(\l)}$  is non-degenerate.  By induction on $n$, the corresponding Euler polynomial has degree \begin{equation}\operatorname{Vol}\Delta_{HD} \cdot n!=L(HD)=\sum_{j=1}^rp_j.
\end{equation}
Let $u_1(\l,p),\cdots, u_{L(HD)}(\l,p)\in \Z_p$ be the list of reciprocal roots of the Euler polynomial in $\overline \Q_p$, the algebraic closure of $\Q_p$. 
Let $r_i(\l,p)=\text{ord}_p u_i(\l,p)$. A priori, these are non-negative rational numbers. For $x_0f_{HD}$, they are integers. 
From these $r_i(\l,p)$ one can compute the Newton Polygon in a way analogous to Definition 1.2 of \cite{Wan-Newton}. 
It was conjectured by Wan in \cite{Wan-Newton} that for generic choice of $\l$ and large prime $p$, the Newton polygon and the Hodge polygon agree. This conjecture was proved by Blache in \cite{Blache11}. This means the Hodge numbers that can be computed combinatorially also encode arithmetic information of the exponential sums. They are also called Hodge-Tate weights of the  datum $HD$.

\begin{example}[${\G=[3,-1,-1,-1]}, \l=2$]
$$Z_p(f_{HD}(2);T)=(1-T)(1-a_p(f_{54.2.a.b}){ p}T+p^{{2}+1}T^2), \quad \forall p>3$$ where $a_p$ denotes the $p$th coefficient, $f_{54.2.a.b}$ is a weight-2 level 54 modular form in its L-Functions and Modular Forms Database label. Thus for a generic $p$, its reciprocal roots in $\C_p$ have $p$-adic orders $ \{0,1,2\}$.
\end{example}

In the literature of hypergeometric functions and motives, there is an explicit way to compute the Hodge numbers using the so-called zig-zag diagram (see \cite{Corti-Golyshev11} by Corti and Golyshev and \cite{RRV22} by Roberts and Rodriguez-Villegas). These Hodge numbers are useful for the determination of the weights of the modular forms according to Galois representations arising from hypergeometric motives, see for example \cite{Allen2024explicit,LLT2} for some applications. In this paper we give two more ways to   explain why:
\begin{quote}
 \emph{The zig-zag diagram algorithm will produce the the Hodge multiplicity and hence Newton multiplicity numbers for $x_0f_{HD}$.} \end{quote} 
 Our approaches below are different from \cite{Corti-Golyshev11}. They are combinatorial and explicit.  

 \medskip

 The paper will be organized in the following manner. The main statement is given in \S \ref{ss:2}, the relation between the zig-zag diagrams and $p$-adic orders of the hypergeometric coefficients will be recalled in \S\ref{ss:3}. In \S \ref{ss:4} and \S \ref{ss:another}, we give two different proofs of the main statement.
 
\section{Main statement}\label{ss:2}
Given a hypergeometric gamma vector $\G=[p_1,\cdots, p_r,-q_1,\dots,-q_s]$, let 
\begin{equation}\label{eq:HD-G}
HD_{\G}=\left \{\alpha=\left\{\bigcup_{j=1}^r \left \{ \frac{i}{p_j} \right \}_{j=1}^{p_i}\right\}, 
\quad \beta=\left\{\bigcup_{j=1}^s  \left \{ \frac{i}{q_j} \right \}_{i=1}^{q_j}\right\} 
\right\}.    
\end{equation}Each multi-set has length $L:=L(HD)=\sum_{i=1}^r p_i$. By construction,  the multiplicity $n_{b,\beta}$ of any element $b=\frac{i}{q_j}<1$ in $\beta$ can be computed by
\begin{equation}\label{eq:nb}
    n_{b,\beta}=\#\{q_k\mid 1\le k\le s,\, bq_k\in\Z\}.
\end{equation}

\medskip
Sort the combined elements in both $\alpha$ and $\beta$ into a new list satisfying 
\begin{equation}\label{eq:mu}
    \mu_1\le\mu_2\le \cdots \le \mu_{2L},
\end{equation}where  $\mu_{2L}=1$ with the following convention: 
\begin{quote}
    if  $a\in \alpha\cap \beta$, then \emph{those copies of $a$'s in $\beta$ all go before those $a$'s in $\alpha$}. 
\end{quote}  Further let $\mu_0=0$.
\medskip

Given $\G$, we define an arithmetic function $\Phi$ recursively as follows:
\begin{equation}\label{eq:Phi}
  \Phi(0)=  \Phi_{HD}(0)=r(\G(HD)), \quad \Phi_{HD}=\Phi_{HD}(i)=\begin{cases}\Phi_{HD}(i-1)+1,& \text{ if } \quad \mu_i\in \alpha;\\ \Phi_{HD}(i-1)-1,& \text{ if } \quad \mu_i\in \beta.
\end{cases}
\end{equation}

\begin{theorem}\label{thm:main}
Let 
\begin{equation}\label{eq:HD-z} 
HT(HD):=\{\Phi(i)\mid \text{ if } \mu_i\in\beta\}. 
\end{equation} They are the Hodge numbers of $HD$.
\end{theorem} 
Equivalently,
\begin{theorem}\label{thm:main}
Given a hypergeometric vector $\G$ as \eqref{eq:gamma-v}, let $HD=HD_\G$ as in \eqref{eq:HD-G}. Then the Hodge numbers $H_{\Delta(HD)}(k)$ of the convex polyhedron $\Delta(HD)$ (to be defined in \eqref{eq:m}) satisfy that 
\begin{equation}\label{eq:H-w}
    H_{\Delta(HD)}(k)= 
        \#\{w\in HT(HD)\mid w=k\}.
\end{equation}
\end{theorem} 
Removing overlapping elements from both $\alpha,\beta$ with the same multiplicity will result in a \emph{derived primitive hypergeometric datum}, denoted by $HD^{\text{red}}$. The zig-zag diagram for $HD^{\text{red}}$ can be obtained by concatenating the  ``V" ditches caused by elements in $\alpha\cap \beta$. Typically, we further normalize the Hodge numbers of $HD^{\text{red}}$ so that the minimal Hodge number becomes 0.

\begin{example}Let $\G=[3,-1,-1,-1]$, then
$HD_\G=\{\alpha=\{\frac13,\frac23,1\},\beta=\{{\color{blue}1,1,1}\}\}.$ The sorted list is
$$\frac13\le \frac23\le {\bl 1} \le {\bl 1}\le {\bl 1}\le 1 $$
\begin{center}
      \includegraphics[height=4cm]{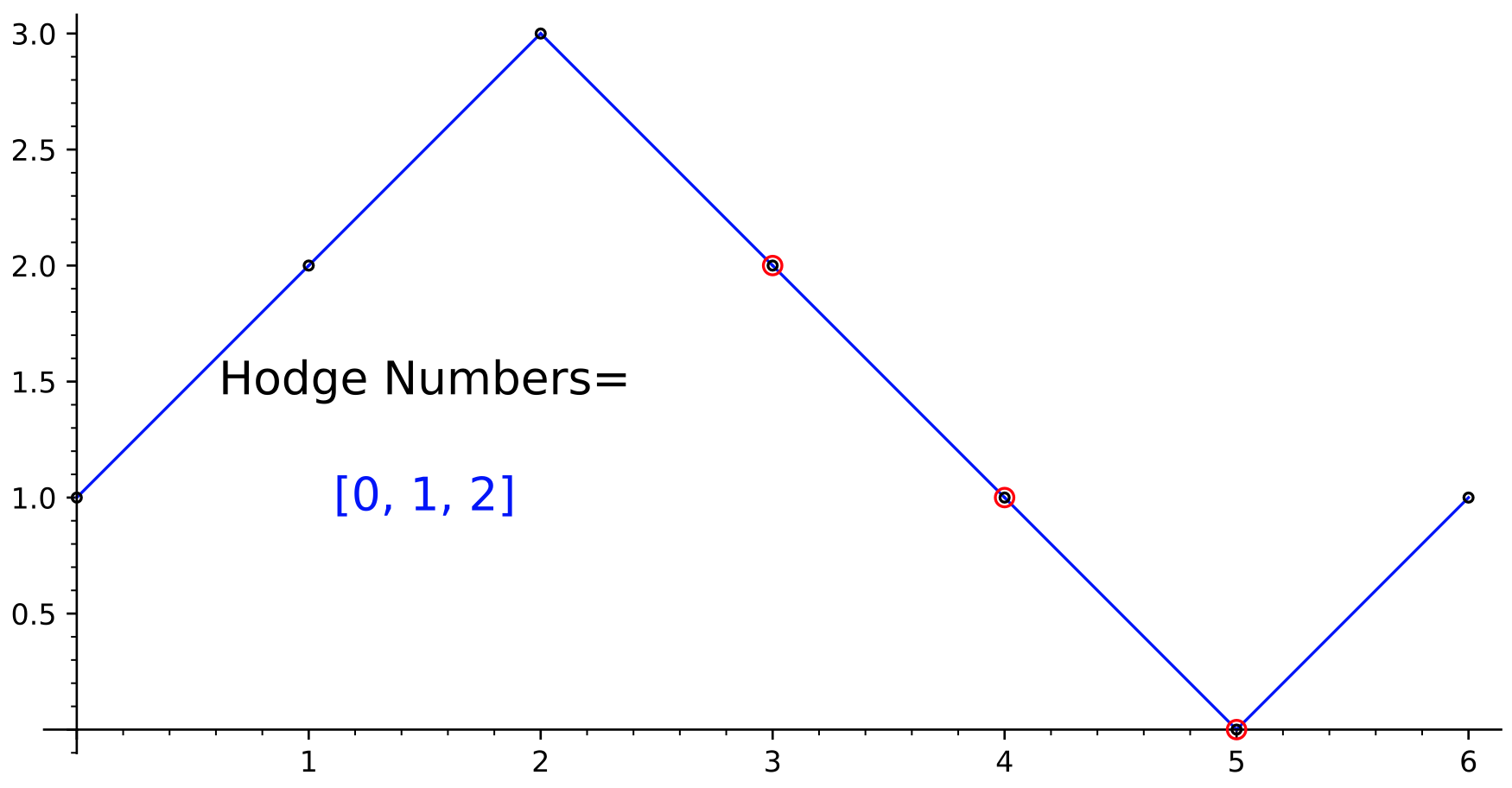}
\end{center} 

   $HD^{\text{red}}=\{\{\frac13,\frac23\},\{1,1\}\}$.

\begin{center}
    \includegraphics[height=4cm]{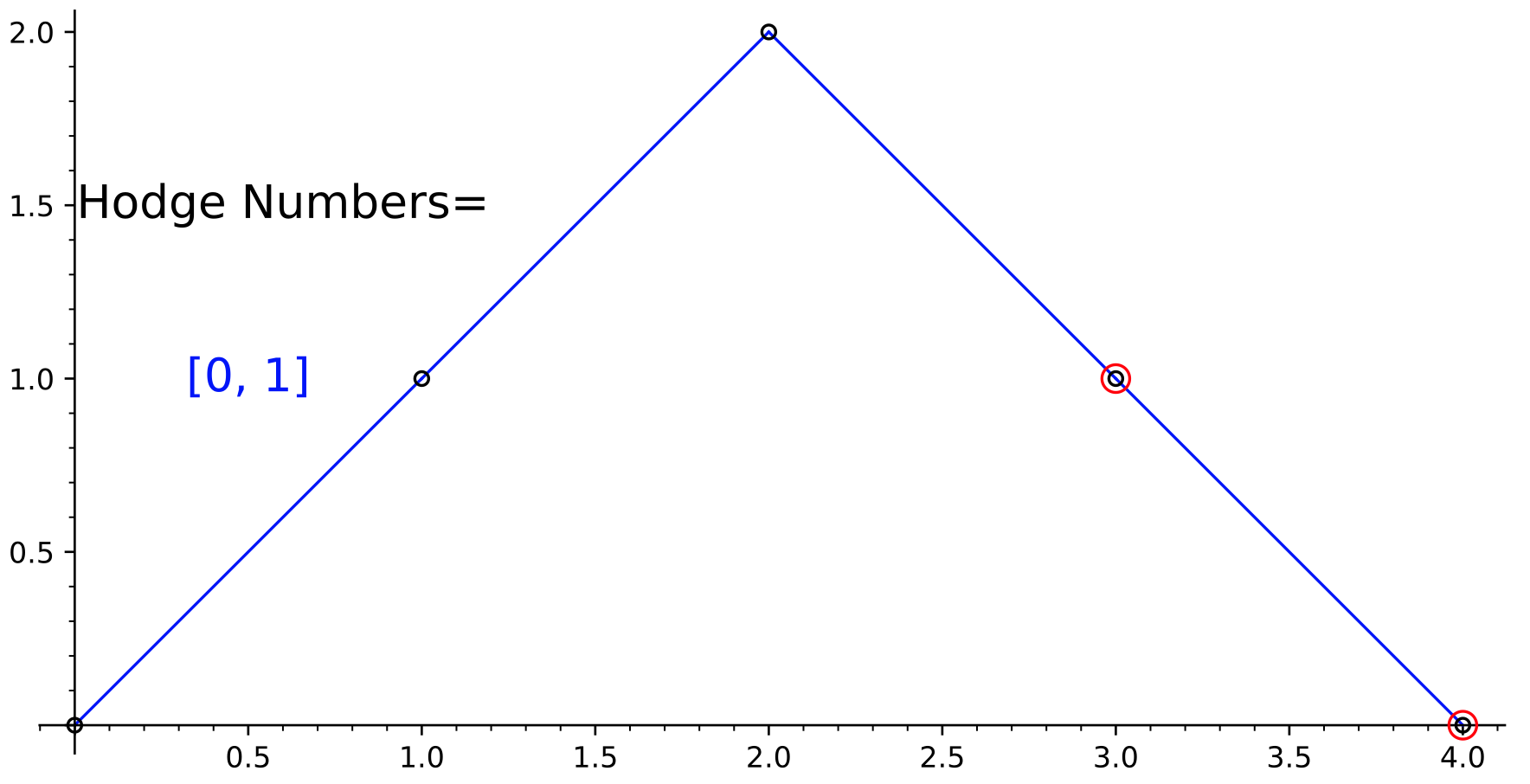}
\end{center}
\end{example}

\section{Zig-zag diagrams and $p$-adic orders}\label{ss:3}
The zig-zag diagram is related to the  $p$-adic order of  
\begin{equation}\label{eq:AHD(k)}\displaystyle
    A_{HD}(k):=\frac{\prod_{a\in\alpha}(a)_k}{\prod_{b\in\beta}(b)_k},
\quad  \text{where} \quad (a)_k:=\frac{\G(a+k)}{\G(a)}=a(a+1)\cdots(a+k-1).\end{equation} It is related to the discussion in \cite{Long18}. Here we recast the idea.  For $a\in \Z_p$,  use $[a]_0$ to denote the first  $p$-adic digit of $a$. If $a=\frac cd$ where $c,d\in \mathbb N, 1\le c<d<p$, and $p\equiv 1\mod d$  then 
\begin{equation}\label{eq:0-digit}
   \left [-\frac{c}{d}\right]_0=\frac{cp- c}d=\frac cd(p-1). 
\end{equation}
 For $a\in (0,1]$ and $1\le k<p$,    \begin{equation}\label{eq:(a)k-porder}
     \text{ord}_p(a)_k=-\left \lfloor a- \frac{ k}{p-1}\right \rfloor=\begin{cases}
    0 & \text{if } 1\le k\le  [-a]_0;\\
    1& \text{if }  [-a]_0< k<p.
\end{cases}
 \end{equation}
Let $p\equiv 1\mod M(HD)$ be a fixed  prime throughout this paper.    It follows that 
\begin{equation}
 0\le   [-\mu_1]_0\le [-\mu_2]_0\le \cdots \le [-\mu_{2L}]_0\le [-1]_0=p-1.
\end{equation}
When $k$ goes from 0 to $p-1$, $\text{ord}_pA_{HD}(k)$ is a step function  starting from 0, and when $k$ passes through  $[-\mu_i]_0$, its value   will change by $n_{\mu_i,\alpha}-n_{\mu_i,\beta}$ where  $n_{\mu_i, \alpha}$ (resp. $n_{\mu_i, \beta}$) denotes the multiplicity of $\mu_i$ in $\alpha$ (resp. $\beta$).  We summarize the above discussion into the next Proposition.

\begin{prop}\label{prop:1}
For a given datum $HD=\{\alpha=\{a_1,\cdots, a_L\},\beta=\{b_1,\cdots,b_L\}\}$   defined over $\Q$ where $a_i,b_j\in (0,1]$, $b_1=1$, whose Gamma vector is $\G(HD)$, let $\mu_1\le \cdots \le \mu_{2L}=1$ be the sorted combined list. For any prime $p\equiv 1\mod M(HD)$ 
\begin{equation}\label{eq:Phi-A}
\Phi_{HD}(i)=\text{ord}_p A_{HD}(k)+r(\G(HD)), \quad \text{if}\quad  [-\mu_i]_0< k \le [-\mu_{i+1}]_0.
\end{equation}
In particular, for each $i$ such that   $\mu_{i}<\mu_{i+1}$, by \eqref{eq:0-digit} and \eqref{eq:Phi-A},
\begin{equation}\label{eq:ordpA-Phi}
\Phi_{HD}(i)=\text{ord}_p A_{HD}(\mu_{i+1} (p-1))+r(\G(HD)).
\end{equation}
\end{prop}
\smallskip

Next we will consider the first term on the right hand side of \eqref{eq:Phi-A} by recalling some useful facts.
For $r\in \mathbb R$, let $\{r\}:=r-\lfloor r\rfloor$. Then
\begin{equation}\label{eq:fracsum}
    \{r\}+\{-r\}=\begin{cases}
        0& r\in \Z;\\
        1& r\notin \Z.
    \end{cases}
\end{equation} 
Also (from the multiplication formula of Gauss sums and the Gross-Koblitz formula), for any  $m\in \Z_{>0}$   $$\{ rm\}+\sum_{i=1}^{m} \left\{ \frac im\right\}=\sum_{i=1}^{m} \left\{ r+\frac i m \right \}. $$  Equivalently,
\begin{equation}\label{eq:fr}
    \{ rm\}=\sum_{i=1}^{m} \left\{ r+\frac i m \right \}-\frac{m-1}2.
\end{equation}

For $1\le k\le  p-1,$ from \eqref{eq:(a)k-porder}
\begin{equation}
    \label{eq:ordpA}
\text{ord}_p A_{HD}(k)=\sum_{i=1}^L -\left \lfloor a_i-\frac{k}{p-1} \right \rfloor + \left \lfloor b_i-\frac{k}{p-1} \right \rfloor.
\end{equation}
In particular, $\text{ord}_p A_{HD}(p-1)=s-r$.

\begin{lemma}\label{lem:1.1}
Given a Gamma vector $\G$, let $HD_\G$ be as in \eqref{eq:HD-G}. For  $p\equiv 1\mod M(HD_\G)$,  $\mu \in \alpha\cup \beta,$ 
 \begin{equation}\label{eq:p-ord-mu(p-1)}
   \text{ord}_p A_{HD_\G}(\mu (p-1))= s -r+\sum_{i=1}^{r}   \{-\mu p_{i}\}- \sum_{j=1}^{s} \{-\mu q_j\}.
 \end{equation}
\end{lemma}
\begin{proof}
\begin{eqnarray*}
\text{R.H.S.}&\overset{\eqref{eq:fr}}=&s-r+  \sum_{i=1}^r \left ( \sum _{l=1}^{p_i} \left \{-\mu +\frac l{p_i}\right\} -\frac{p_i-1}2\right) -  \sum_{j=1}^s \left ( \sum _{l=1}^{q_i} \left\{ -\mu +\frac l{q_j}\right\} -\frac{q_j-1}2\right)\\
&=&\frac{s-r}2+  \sum_{i=1}^r  \sum _{l=1}^{p_i} \left ( -\mu +\frac l{p_i}- \left \lfloor -\mu +\frac l{p_i}\right \rfloor \right) -  \sum_{j=1}^s  \sum _{l=1}^{q_i} \left (- \mu +\frac l{q_j}- \left \lfloor -\mu +\frac l{q_j}\right \rfloor \right) \\
&=&  \sum_{i=1}^r \sum _{k=1}^{p_i}    - \left \lfloor \mu -\frac {k}{p_i}\right \rfloor  +  \sum_{j=1}^s   \sum _{k=1}^{q_j}     \left \lfloor \mu -\frac {k}{q_i}\right \rfloor\\ 
   &\overset{\eqref{eq:ordpA}}=&\text{ord}_p A_{HD_\G}(\mu  (p-1))\in \Z.
\end{eqnarray*} 
\end{proof}

\section{Hodge numbers of convex cones}\label{ss:4}
\subsection{Background} The background for this section is from \cite{Wan-Newton} by Wan. 
Let $\Delta$ be a polyhedron in $\R^n$ whose vertices are  points in $\Z^n$. For each $\v\in\R^n$,   its  \emph{depth}, denoted by $dep(\v)$ is the sum of its negative entries. For the computation below we assume the interior of $\Delta$ contains no lattice point (vectors in $\Z^n)$ so that its  denominator $1$. Consider the cone $C(\Delta):=C(\Delta,\mathbb R_{\ge 0})$ obtained from rays starting from the origin pointing to points in $\Delta$. For each lattice point $\v$ in this cone, the \emph{weight} of $\v$, denoted by $wt(\v)$ is given as follows: $wt({\bf o})=0$, otherwise $wt({\v})$ is the non-negative scalar $k$ such that $\bf v$ is located on one of the outer faces of the convex polyhedron obtained from $\Delta$ re-scaled by $k$. 
Given   $k\in\Z_{\ge0}$,  let $w_{\Delta}(0)=1$ and for $k\ge 1$, $$w_{\Delta}(k)=\#\left\{ \v \in  C(\Delta) \cap \Z^n \mid wt (\v)=k\right\}.$$  
Then the Hodge numbers of the cone is computed by the inductive formula
\begin{equation}\label{eq:m}
     H_{\Delta}(k):=\sum_{i=0}^n (-1)^i \binom{n}{i} w_{\Delta}(k-i).
 \end{equation}
And $ H_{\Delta}(k)=0$ for $k>n$, see \cite{Wan-Newton}. \smallskip

Now let $\Delta(HD)$ be the polyhedron with vertices $\{{\bf o}\}\cup \{\e_i\}_{i=0}^{n}$, where $\e_i$ as in \eqref{eq:ei}. It  has dimension $n$, and  no interior lattice point. For any lattice point $\v$ of $C(\Delta)$, $wt(\v)$ is the value of its first entry.

\subsection{Primitive elements and a generating set}
Let $$G:=C(\Delta(HD))\cap \Z^n=\left\{\sum _{i=0}^n c_i \e_i \mid c_i\in \mathbb R_{\ge 0} \right \}\cap \Z^n,$$ which is  a semi-group. It is larger than $\{\sum_{i=0}^n a_i\e_i, a_i\in \Z_{\ge 0}\}$  as $\{\e_i\}_{i=0}^n$ is  linearly dependent. Let $$H:=C(\{\e_0,\cdots, \e_{n-1}\},\Z_{\ge0})=C(\{\e_0,\cdots, \e_{n-1}\},\Q_{\ge0})\cap \Z^n.$$ 
Any  $\v\in G$ is said to be \emph{primitive} if  its depth is negative.  
In this case $\v=\sum_{i=0}^n c_i \e_i$ where $c_i\in \Q_{\ge 0}$, $c_n>0$. 
 As $\v$ lies in one of the faces  of $\Delta(HD)$ re-scaled by $k$, we may assume  $c_i=0$ for one $i\in[0,n]$. Then $\{\e_j\}_{j=0,j\neq i}^n$ form an invertible matrix whose determinant is either $\pm p_j$  or $\pm q_j$. For more details, see Lemma \ref{lemma: determinant} and its proof. Thus for $l\neq i$,  $c_l=\frac{k_l}{p_j}$ or $\frac{k_l}{q_j}$ for some positive integers $k_l,j$. 
 For any $\mu \in (\alpha \cup \beta)$, let
 \begin{equation}\label{eq:ceil}
   \u_\mu:=\mu \e_n+ \left \{ -\mu p_r\right\}\e_0+ \sum_{i=1}^{r-1}\{- \mu p_i\}\e_i + \sum_{j=1}^{s-1}\{ \mu q_j\} \e_{r-1+j}.
\end{equation}In particular when $\mu=1$, $\u_1=\e_n$. It  has weight 1 and depth $-L$. When $\mu<1$, 
\begin{equation}\label{eq:wt-ceil}
\begin{split}
    wt(\u_\mu)&= \sum_{i=1}^{r}\{-\mu p_i\} + \sum_{j=1}^{s}\{ \mu q_j\}\\
    &\overset{\eqref{eq:fracsum}}= \sum_{i=1}^{r}\{-\mu p_i\}  -\sum_{j=1}^{s}\{- \mu q_j\}+s-\#\{j\mid 1\le j\le s, \mu q_j\in\Z\}\\
&\overset{\eqref{eq:nb},\eqref{eq:p-ord-mu(p-1)}}=r+\text{ord}_p A_{HD}(\mu (p-1)) -\#\{ \mu  \in \beta\}.
\end{split}
\end{equation}
From construction,  $\u_\mu$ is in $ G$.
Intuitively, it is the ``ceiling" function of $\mu\e_n$, namely the element in $G$ closest to the vector $\mu\e_n$ with $dep(\u_\mu-\mu\e_n)=0$. 

Assume the $i$th entry of the combined list \eqref{eq:mu} is in $\beta$, namely $\mu_i\in\beta$ with multiplicity $n_{\mu_i,\beta}$ such that all of its other duplicates in $\beta$ have smaller indices. 
Thus $\mu_{i-n_{\mu_i,\beta}}<\mu_{i-n_{\mu_i,\beta}+1}=\mu_i$. Then
\begin{equation}\label{eq:wt-mu}
    wt(\u_{\mu_i})=wt(\u_{\mu_{i-n_{\mu_i,\beta}+1}})\overset{\eqref{eq:wt-ceil}}=r+\text{ord}_p A_{HD}( \mu_{i-n_{\mu_i,\beta}+1}(p-1)) -n_{\mu_i,\beta} \overset{\eqref{eq:ordpA-Phi}}=\Phi(i-n_{\mu_i,\beta})-n_{\mu_i,\beta}\overset{\eqref{eq:Phi}}=\Phi(i).
\end{equation}Also for $0\le l\le n_{\mu_i,\beta}-1$
\begin{equation}\label{eq:wt-mu-before}
   \Phi(i-l)\overset{\eqref{eq:Phi}}=\Phi(i)+l.
\end{equation}
We now focus on $\u_{\mu}$ where $\mu\in \beta$. For convenience, we let $b_0=0, b_{L+1}=2$. They are not in $\beta$.
\begin{lemma}\label{lem:1}Assume the elements in $\beta$ are listed in a non-decreasing order, i.e. $0=b_0<b_1\le b_2\le \cdots \le b_L< b_{L+1}$.  

1) For each positive $b\in \beta$, $\u_b$ is primitive.

2) For $b_i\neq b_j$ in $\beta$,  $\u_{b_i}\neq \u_{b_j}$. 

3) Assume $b_{i-1}<b_{i}$ and let $S_i=\{1\le j\le s-1 \mid b_i q_j\in \Z\}$, then $\u_{b_i}+\sum_{j\in S_i}\e_{r+j}\in \u_{b_{i-1}}+\, H$
\end{lemma}
\begin{proof}
Note that the  $(r-1+j)$th entry of
$\u_{b_i}$ is $-\lfloor b_i q_j \rfloor$ and at least one of them is negative, so $\u_{b_i}$ is primitive. If $b_i\neq b_j$,  by \eqref{eq:ceil} $\u_{ b_i}\neq  \u_{b_j}$.

Each time when $i$ increases by 1, the $(r-1+j)$'s entry of $\u_{b_i}$  either stays the same or  decrease by 1 depending on whether $u_iq_j\in\Z$. 
\end{proof}

\begin{lemma}\label{lem:2}
For each primitive  ${\bf v}\in G$, there is exists an element 
$b\in \beta$, such that ${\bf v}\in \u_b+ \,  \langle H,\e_n\rangle=\{i\u+j\e_n\mid \u \in H, i,j\in\Z_{\ge0}\}.$ 
\end{lemma}
\begin{proof}Assume  ${\bf v}=\sum_{i=0}^{n} c_i \e_i $ and $c_n>0$. If $c_n\in \Z$ and then we can pick $b=1$; otherwise, replace $\bf v$ by ${\bf v}-\lfloor c_n \rfloor \e_n$ we may assume $c_n \in (0,1)\cap \Q$. Next pick the largest element $b\in \beta$, satisfying  $b\le c_n$. Namely pick the largest value among fractions of the form $\frac{\lfloor c_n q_j\rfloor}{q_j}$ in $\beta$. Up to relabelling, we may assume $b=\frac{\lfloor c_n q_1\rfloor}{q_1}$. The $r-1+j$ entry of $\u_{ c_n}$ (resp. $\u_{ \frac{\lfloor c_n q_1 \rfloor}{q_1}}$) is 
$-\lfloor c_n q_j\rfloor$ (resp. $-\lfloor \frac{\lfloor c_n q_1 \rfloor}{q_1} q_j\rfloor$). From the choice of $b$, either $\lfloor c_n q_j\rfloor=\lfloor \frac{\lfloor c_n q_1 \rfloor}{q_1} q_j\rfloor$, or $\lfloor c_n q_j\rfloor=\lfloor \frac{\lfloor c_n q_1 \rfloor}{q_1} q_j\rfloor+1$. In the latter case, $$\frac{\lfloor c_n q_j\rfloor}{q_j}= \frac{\lfloor \frac{\lfloor c_n q_1 \rfloor}{q_1} q_j\rfloor+1}{q_j}\le  \frac{\lfloor c_n q_1 \rfloor}{q_1}$$ which implies $1-\left \{  \frac{\lfloor c_n q_1 \rfloor}{q_1} q_j \right\}<0$, a contradiction. Thus $\u_{ c_n}$ and  $\u_{ \frac{\lfloor c_n q_1 \rfloor}{q_1}}$ have the same $r-1+j$ entries for $j=1,\cdots, s-1$. Also for each $2\le i\le r$, the $i$th entry of $\u_{c_n}$ is greater or equal to the $i$th entry of $\u_{ \frac{\lfloor c_n q_1 \rfloor}{q_1}}$. Thus ${\bf v}\in \u_b+ \,  \langle H,\e_n\rangle.$ 
\end{proof}

\subsection{A set of primitive elements}
We now modify the list $\{\u_b\mid b\in \beta\}$ by incorporating the multiplicities of the elements to get another set consisting of distinct elements. For each $\mu_i\in\beta$,  let $S_i=\{1\le j\le s-1 \mid \mu_i q_j\in \Z\}$ as in Lemma \ref{lem:1}, let $n_{\mu_i,\beta}=|S_i|$. Further we let $\displaystyle {\bf v}_i:=\sum_{j\in S_i} \e_{r-1+j}$ and let ${\bf v}_{i,l}$ be the partial summand of ${\bf v}_i$ consisting the first $l$ entries, $wt(\v_{i,l})=l$. 

Like before, we assume all of its other duplicates of $\mu_i$ in $\beta$ appear before $\mu_i$. For each $1\le l<m_i$,   
\begin{equation}\label{eq:w}
    w_{dep(\u_{\mu_i})-l}:= \u_{\mu_i}+\v_{i,l}.
\end{equation}\
It has weight $wt(\u_{ b_i})+l=\Phi(i)+l=\Phi(i-l)$ by \eqref{eq:wt-mu} and \eqref{eq:wt-mu-before}, and depth $dep(\u_{\mu_i})+l$. 

Further we let $\w_0={\bf o}$.  In conclusion, we have
\begin{prop}\label{prop:wt-Phi}
For each $0\le i\le L$,  $dep(\w_i)=-i$, and
  \begin{equation}
     \{wt(\w_i)\}_{i=1}^L=\{\Phi(i) \mid 1\le i\le 2L, \mu_i\in \beta\}. 
  \end{equation}
\end{prop}

\begin{example}
 $\v=[4,4,2,-3,-6,-1]$
 
$$  \begin{array}{|c|l|l|c|}\hline
b&\u_b&\w_i\\ \hline
  \frac{1}{6}&[2, 1, 1, 0, -1]&\w_1=[\bl 2\bk, 1, 1, 0, -1]\\
\frac{1}{3}&[2, 2, 2, -1, -2]&\w_2=[\bl3\bk, 2, 2, 0, -2]\\
 \frac{1}{2}&[1, 2, 2, -1, -3]&\w_3=[\bl2\bk, 2, 2, -1, -2]\\
\frac{2}{3}&[2, 3, 3, -2, -4]&\w_4=[\bl1\bk, 2, 2, -1, -3]\\
\frac{5}{6}&[3, 3, 3, -2, -4]&\w_5=[\bl3\bk, 3, 3, -1, -4]\\
1&[1,4,4,-3,-6]&\w_6=[\bl2\bk, 3, 3, -2, -4]\\
&&\w_7=[\bl3\bk,4,4,-2,-5]\\
&&\w_8=[\bl2\bk,4,4,-2,-6]\\
&&\w_9=[\bl1\bk,4,4,-3,-6]\\ \hline
 \end{array}$$ 
In this case $HD_\G=\{\{\frac14,\cdots,\frac44,\frac14,\cdots,\frac44,\frac12,1\},\{\frac13,\frac23,\frac33,\frac16,\cdots,\frac66,1\}\} $ 
 \begin{center}
  \includegraphics[width=5in]{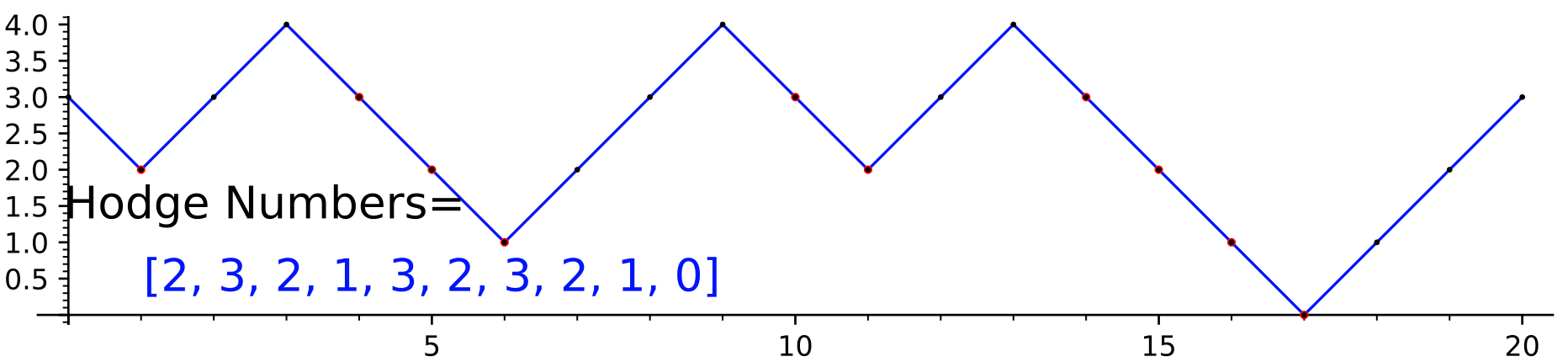}   
 \end{center}
\end{example}

\begin{prop}\label{prop:partition}
    
For each $i\in[0,L]$ there is an integer $s(i)\in\{1,\cdots,n\}$ such that \begin{equation}\label{eq:key}
   G=
   \bigsqcup_{i=0}^{L} \left  (\w_i+ C(\{\e_0,\e_1,\cdots, \check{\e}_{s(i)},\cdots, \e_n\}, \Z_{\ge 0}) \right) , 
\end{equation} where $\w_i$ is defined in \eqref{eq:w}, $\check{\e}_{s(i)}$ means removing ${\e}_{s(i)}$ from the list.
\end{prop}

\begin{proof} 

We first prove that the right-hand side is a disjoint union. We let $s(0)=n$. So $\w_0+ C(\{\e_0,\e_1,\cdots, \check{\e}_{n-1}\}, \Z_{\ge 0})=\, H$.  Since $\w_1$ has depth $-1$, there exists a unique $\e_i$, whose index is $s(1)$  such that $\w_1+\e_{s(1)}$ is in $ H$. Thus $ H$ and $\w_1+ C(\{\e_0,\e_1,\cdots, \check{\e}_{s(1)},\cdots, \e_n\}, \Z_{\ge0})$ are disjoint. Proceed by induction and assume   the sets $$U_{k-1}:=\,   H \bigsqcup_{i=1}^{k-1} \left  (\w_i+ C(\{\e_0,\cdots, \check{\e}_{s(i)},\cdots, \e_n\}, \Z_{\ge0}) \right)$$ is  a disjoint union, as claimed. For  $\w_{k}$, by construction, there exists a unique $s(k)$ such that $\w_{k+1}+\e_{s(k)}$ is in $U_{k-1}$.  Thus $\w_{k}+ C(\{\e_0,\cdots, \check{\e}_{s(k)},\cdots, \e_n\}, \Z_{\ge0})$ is disjoint from $U_{k-1}$ and let $U_k=U_{k-1} \bigsqcup (\w_{k}+ C(\{\e_0,\cdots, \check{\e}_{s(k)},\cdots, \e_n\}, \Z_{\ge0})).$

Next we show that every primitive element $\v$ lies in one of the cones with vertices $\w_i$. Taking away a positive integer multiple of $\e_n$ if necessary, we may assume the depth of $\v$ is larger than $-L$. By the previous Lemma, $\v=\u_{b_i}+\w$ for some $b_i<1\in \beta$ and $\w \in\,  H$. Assume now $b_i$ is the least among $\beta$ such that $\v-\u_{b_i}$ in $ H$, by 3) of Lemma \ref{lem:1}  this means $\w-(\sum_{i\in S_i} \e_{r+i})\not \in \,  H$ so it lies in one of the cone $\w_i + C(\{\e_0,\e_1,\cdots, \check{\e}_{s(i)},\cdots, \e_n\}, \Z_{\ge0}).$
\end{proof}

The previous proposition gives a partition of $ G$ into lattices points located in $L+1$ cones. Each of them is generated by a  polyhedron, denoted by $\Delta_i$, with vertices $\{{\bf o},\e_0,\e_1,\cdots, \check{\e}_{s(i)},\cdots, \e_n\}$  whose apex located at $\w_i$.
\begin{example}For $\G=[5,-2,-2,-1]$, $
\w_0=[0,0,0],\w_1=[2,0,-1], s(1)=2;
\w_2=[1,-1,-1], s(2)=1;
\w_3=[2,-1,-2], s(3)=2;
\w_4=[1,-2,-2],s(4)=1.$ In the following plot, we mark the lattices in $\w_i+C(\Delta_i,\Z_{\ge0})$ with colors green, red, purple, blue and orange respectively.
\begin{center}
    \includegraphics[height=3.5in]{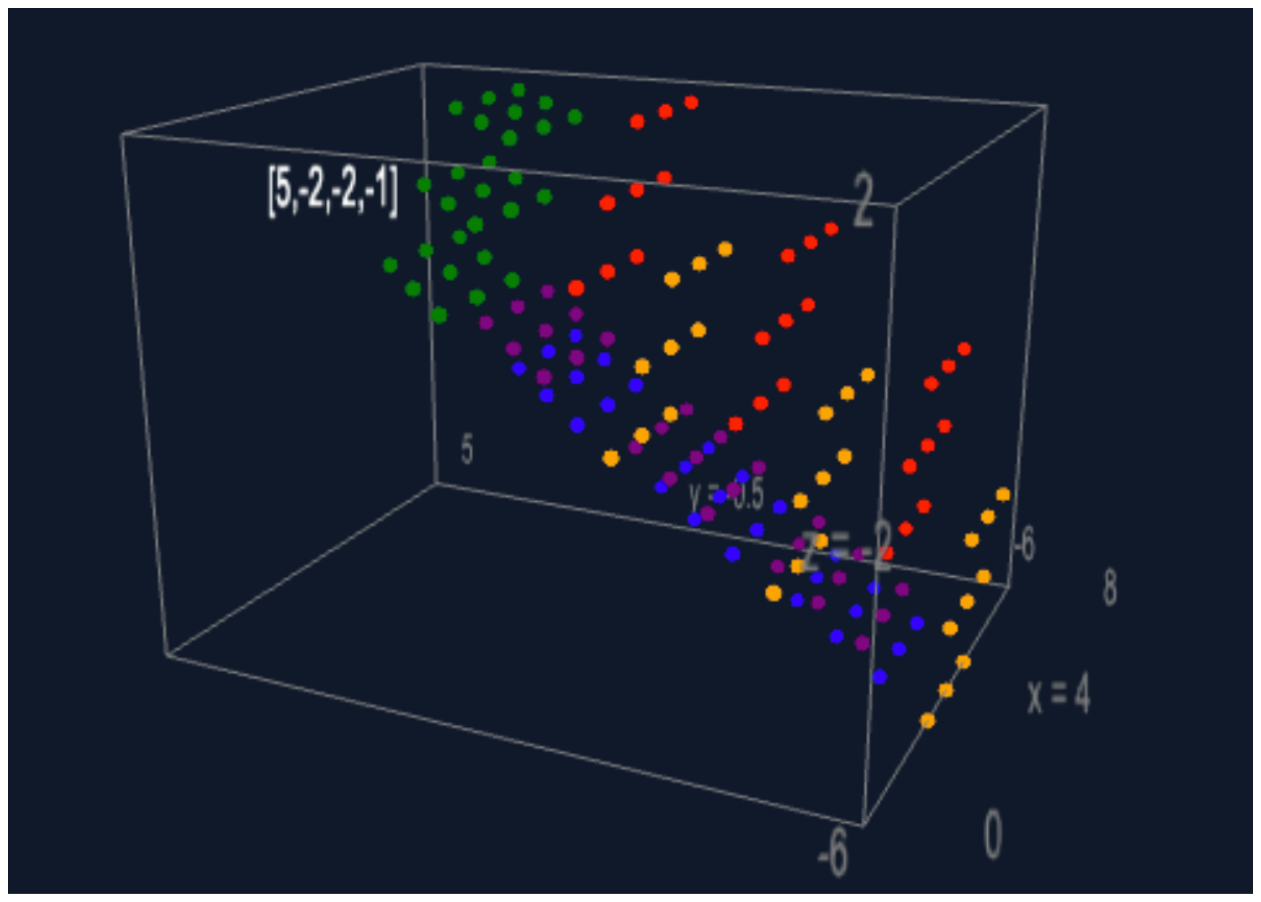}
\end{center}
\end{example}

\medskip
The following identity  follows from the Chu-Vandermonde's identity that 
for any  $n,k\in \Z_{\ge0},$
\begin{equation}\label{eq:eval}
\sum_{i=0}^n (-1)^i \binom{n}{i} \binom{n-1+k-i}{n-1} =\begin{cases} 1& \text{ if} \quad k=0;\\0&\text{else}.\end{cases}    
\end{equation}

\begin{prop}\label{prop:H-wt}
For $0\le k\le n=r+s-1$,
$H_{\Delta(HD)}(k)=\#\{\w_i \mid 0\le i\le L-1, wt(\w_i)=k\}.$
\end{prop}
\begin{proof}When $k=0$, both hand sides take value 1. 
The number of weight-$k$ elements in each $\w_i+C(\Delta_i,\Z_{\ge 0})$ is the number of weight $k-wt(\w_i)$ elements in $C(\Delta_i,\Z_{\ge 0})$. For each  $\Delta_i$, $\{\e_0,\e_1,\cdots, \check{\e}_{s(i)},\cdots, \e_n\}$ are linearly independent in $\mathbb R^n$ and each  has weight 1. Thus
$$w_{\Delta_i}(k):=\#\{\w\in \w_i+C(\Delta_i,\Z_{\ge 0})\mid wt(\w)=k\} = \binom{n-1+k-wt(\w_i)}{k-wt(\w_i)}.$$

\begin{eqnarray*}
    H_{\Delta(HD)}(k)&\overset{\eqref{eq:m}}=&\sum_{i=0}^n (-1)^i\binom{n}{i}w_{\Delta(HD)}(k)\\
    &=&\sum_{i=0}^n (-1)^i\binom{n}{i}\sum_{i=0}^L w_{\Delta_i}(k)\\
    &=&\sum_{i=0}^L \sum _{i=0}^n (-1)^i\binom{n}{i}  \binom{n-1+k-wt(\w_i)}{k-wt(\w_i)}\\
    &\overset{\eqref{eq:eval}}=& \#\{\w_i\mid 0\le i\le L, wt(\w_i)=k\}.
\end{eqnarray*}
\end{proof}

The proof of Theorem \ref{thm:main} follows from putting together Propositions \ref{prop:wt-Phi} and \ref{prop:H-wt}.

\section{Another proof}\label{ss:another}

Let $\G=[p_1,\cdots,p_r,-q_1,\cdots,-q_s]$ be a Gamma vector. Below we modify the notation by letting  $p_{r-1+j}=q_j$ for $1\le j\le s$ with $p_{r+s-1}=1$. 
Set $n=r+s-1$. Let $\{\e_0,\ldots,\e_{n}\}$ as before. Note that 
$$
L=p_1+\cdots+p_r=q_1+\cdots+q_s
$$ is equivalent to the single relation satisfied by the vectors $\e_i,0\le i\le n.$
\begin{equation} \label{eq: relation}
p_1\e_1+\cdots+p_{r-1}\e_{r-1}=p_{r}\e_{r}+\cdots+p_{n}\e_{n}.
\end{equation}

We let $\Delta=\Delta(HD)$ denote the $n$-dimensional polyhedron with the origin $\mathbf o$
and $\e_j$, $j=0,\ldots,n$, as vertices. Furthermore,
for a nonempty subset $T$ of
$$
S:=\{0,1,\ldots,n\},
$$
we let $\Delta_T$ denote the 
$(|T|-1)$-dimensional polytope with vertices $\e_j$, where $j\in T$.
Also, we let $C(\Delta_T)$ denote the cone $C(\Delta_T):=\{\sum_{j\in
  T}a_j\e_j:a_j\ge 0\}$. The notation $C(\Delta)$ is similarly defined.
Recall that a nonzero vector $v$ in $C(\Delta_T)$ is said to have \emph{weight} $w$
if $w$ is the smallest real number such that $v\in w\Delta_T$.
Since all $\e_j$ are lying on the hyperplane $x_1=1$, the weight of $v$
is simply the first coordinate of $v$. In particular, even though a
vector $v$ may belong to several $C(\Delta_T)$, the weight of $v$
does not depend on $T$.

For a non-negative integer $k$, let $m(k;T)$ (respectively,
$m(k)$) denote the number of lattice points (points in $\Z^n$) of
weight $k$ in the cone $C(\Delta_T)$ (respectively,
$C(\Delta)$) and \begin{equation}\label{eq:g}
   g(x;T)=\sum_{k=0}^\infty m(k;T)x^k
\end{equation} (respectively,
$g(x)$) be the generating function of $m(k;T)$ (respectively, $m(k)$).

Let
$$
S_+=\{0,\ldots,r-1\}, \qquad
S_-=\{r,\ldots,r+s-1\}.
$$

\begin{Lemma}
  We have
  \begin{equation*}
    \begin{split}
      g(x)&=\sum_{k=0}^{r-1}(-1)^{r-1-k}\sum_{T\subset S_+,|T|=k}
      g(x;T\cup S_-) \\
      &=\sum_{k=0}^{s-1}(-1)^{s-1-k}\sum_{T\subset S_-,|T|=k}
      g(x;S_+\cup T).
    \end{split}
  \end{equation*}
\end{Lemma}

\begin{proof} Here we prove only the first formula; the proof of the
  second formula is similar.
  
  We claim that each nonzero vector $v$ in $C(\Delta)$ has a unique
  expression $v=\sum_{j\in T}a_j\e_j$, $a_j>0$ for all $j\in T$, for
  some subset $T$ of $S$ such that at least one of $0,\ldots,r-1$ is
  missing from $T$.

  We first prove the existence of such an expression.
  Assume that $v$ is a nonzero vector in $C(\Delta)$. Then $v\in
  C(\Delta_{T'})$ for some subset $T'$ of $S$ of cardinality $\le n$,
  say, $v=\sum_{j\in T'}a_j\e_j$, $a_j>0$. If $S_+\not\subseteq T'$,
  then we are done. If $S_+\subseteq T'$, then setting
  $c=\min_{j\in S_+}a_j/p_j$ and using \eqref{eq: relation}, we may
  write $v$ as 
  $$
  v=\sum_{j\in S_+}(a_j-cp_j)\e_j+\sum_{j\in S_-}(a_j+cp_j)\e_j
  $$
  (we set $a_j=0$ in the second sum when $j\notin T'$).
  This is an expression for $v$ with the claimed property.
  The uniqueness follows immediately from the facts that any subset of
  $S$ of cardinality $\le n$ is linearly independent and that all the
  coefficients $b_j$ for $j\in S_+$ in a nontrivial relation
  $\sum_{j\in S}b_j\e_j=0$ have the same sign.

  Having proved the claim, for a nonzero vector $v$ we let
  $\operatorname{supp}(v)$ denote the set $T$ in the claim.
  We also set $\operatorname{supp}({\bf o})=\{\}$. We now prove the
  formula.

  Given a nonzero vector $v$ in $\Z^n\cap C(\Delta)$, we write its
  support as $\operatorname{supp}(v)=T_+\cup T_-$, where
  $T_\pm=S_\pm\cap\operatorname{supp}(v)$. Let $m=|T_+|$. We need to
  show that the vector $v$ is counted exactly once in the expression
  $$
  \sum_{k=0}^{r-1}(-1)^{r-1-k}\sum_{T\subset S_+,|T|=k}
      g(x;T\cup S_-).
  $$
  Indeed, the number of times $v$ is counted in the expression is    
  $$
  \sum_{k=m}^{r-1}(-1)^{r-1-k}\binom{r-m}{k-m}
  =(-1)^{r-m-1}\sum_{k=0}^{r-1-m}(-1)^k\binom{r-m}k=1.
  $$
  This proves the formula.
\end{proof}

We now determine the generating functions $g(x;T)$.

\begin{Lemma} \label{lemma: determinant}
  For $i\in S$, let $M$ be the $n\times n$ matrix whose
  rows are $v_j$, $j\in S$, $j\neq i$. Then $|\det M|=p_j$.
\end{Lemma}

\begin{proof}
  If $i=0$, after applying suitable elementary row operations, we see
  that $\det M$ is equal to the determinant of the matrix formed by
  $\e_1,\ldots,\e_{n-1}$ and
  $(1-p_1-\ldots-p_{r-1}+p_{r}+\cdots+p_{n-1},0,\ldots,0)=(p_r,0,\ldots,0)$.
  Thus, $|\det M|=p_r$. Likewise, if $i=1,\ldots,n$, then $\det M$
  is equal to the determinant of the matrix formed by
  $\e_0,\ldots,\e_{i-1},\e_{i+1},\ldots,\e_{n-1}$, and a vector of the form
  $(\ast,0,\ldots,0,\pm p_i,0,\ldots,0)$. Therefore, $|\det M|=p_i$.
\end{proof}

\begin{Lemma} \label{lemma: denominator}
  Let $T$ be a proper subset of $S$. Assume that
  $v\in\Z^n\cap C(\Delta_T)$, say, $v=\sum_{j\in T}a_j\e_j$. Then
  $a_j\in\frac1{d_T}\Z$ for all $j\in T$, where
  $d_T=\gcd\{p_i:i\in S\setminus T\}$.
\end{Lemma}

\begin{proof}
  For $i=1,\ldots,n+1$, let $T_i=S\setminus\{i\}$. By Lemma
  \ref{lemma: determinant}, if $T=T_i$ for some $i$, then for
  $\sum_{j\in T_i}a_j\e_j\in\Z^n\cap C(\Delta_{T_i})$, we have
  $a_j\in\frac1{p_i}\Z$ for all $j$ and the statement holds.

  More generally, if $T$ is a proper subset of $S$ and $v=\sum_{j\in
    T}a_j\e_j$ is a lattice point contained in $C(\Delta_T)$, then $v$
  is contained in $T_i$ for every $i$ in $S\setminus T$. Then by the
  result above, $a_j\in\frac1{p_i}\Z$ for all $j\in T$ and all $i\in
  S\setminus T$. 
  Therefore, $a_j\in\frac1{d_T}\Z$ for all $j$.
\end{proof}

In the following, for a positive integer $d$ and an integer $m$, we
let $\{m\}_d$ denote the residue of $m$ modulo $d$.
Also, for a real number $x$, we let $\{x\}:=x-\lfloor
x\rfloor$ as before. For a rational number $t$, we define
$$
w(t):=\sum_{j\in S_+}\{-p_jt\}+\sum_{j\in S_-}\{p_jt\}.
$$
We remark that by the assumption that $\sum_{j\in S_+}p_j=\sum_{j\in
  S_-}p_j$, the function $w(t)$ is integer-valued.

\begin{Lemma} \label{lemma: f(x;T)}
  Let $T$ be a proper subset of $S$ and $d:=d_T$ be the
  integer defined in Lemma \ref{lemma: denominator}. Then we have
  $$
  g(x;T)=\frac{f(x;T)}{(1-x)^{|T|}},
  $$
  where
  \begin{equation} \label{eq: f(x;T)}
  f(x;T)=\sum_{c=0}^{d-1}x^{w(c/d)}.
  \end{equation}
\end{Lemma}

\begin{proof}
  We call an element $u$ of $\Z^n\cap C(\Delta_T)$
  to be $T$-primitive (or simply primitive) if $u-\e_j\notin\Z^n\cap C(\Delta)$ for any $j\in T$, so
  every element $v$ of $\Z^n\cap C(\Delta_T)$ can be written uniquely
  as $v=u+\sum_{j\in T}a_j\e_j$ for some $T$-primitive element $u$ and some
  non-negative integers $a_j$. Then we have
  $$
  g(x;T)=\frac1{(1-x)^{|T|}}\sum_{u\text{ $T$-primitive}}x^{\text{weight
      of }u}.
  $$
  Now we determine the $T$-primitive elements of $\Z^n\cap C(\Delta_T)$
  and prove the lemma.

  We first consider the case $n\notin T$. Since $p_{n}=1$, we have
  $d_T=1$. Thus, $f(x;T)=1$. On the other hand, by Lemma
  \ref{lemma: denominator}, $\bf o$ is the only $T$-primitive element, so
  $g(x;T)=1/(1-x)^{|T|}$ and the lemma is proved in this case.

  We next consider the case $0,n\in T$. Let $u=\sum_{j\in T}a_j\e_j$
  be a $T$-primitive element. By Lemma \ref{lemma: denominator},
  $da_j$ are integers for all $j\in T$. Let $c:=da_{n}$.
  The $T$-primitive condition implies that $0\le c\le d-1$. Furthermore,
  the assumption that $u\in\Z^n$ implies that
  $$
  \begin{cases}
    a_j+p_ja_{n}\in\Z, &\text{if }j\in T_+,j\neq 0,\\
    a_j-p_ja_{n}\in\Z, &\text{if }j\in T_-,
  \end{cases}
  $$
  and
  $$
  a_0+\sum_{j\in T,j\neq 0}a_j\in\Z.
  $$
  Thus,
  $$
  da_j=\begin{cases}
    \{-p_jc\}_d, &\text{if }j\in T_+,j\neq 0, \\
    \{p_jc\}_d, &\text{if }j\in T_-,
  \end{cases}
  $$
  and
  $$
  da_0=\left\{c\sum_{j\in T_+,j\neq 0}p_j
    -c\sum_{j\in T_-}p_j\right\}_d.
  $$
  This shows that every primitive element is uniquely determined by
  its coefficient of $\e_{n}$. It shows also that for each integer
  $c$ with $0\le c\le d-1$, there is a $T$-primitive element such that the
  coefficient of $v_{n}$ is $c/d$. We now compute their weights.
  
  From the computation above, we see that the weight of the $T$-primitive
  element $\sum_{j\in T}a_j\e_j$ with $a_{n+1}=c/d$ is equal to
  $$
    \sum_{j\in T}a_j=\frac1d\left(
      \left\{c\sum_{j\in T_+,j\neq1}p_j
        -c\sum_{j\in T_-}p_j\right\}_d+\sum_{j\in T_+,j\neq 0}
      \{-p_jc\}_d+\sum_{j\in T_-}\{p_jc\}_d\right).
  $$
  Observing that $d|p_j$ for any $j\notin T$, we may rewrite the
  expression above as
  $$
    \sum_{j\in T}a_j=\frac1d\left(
    \left\{c\sum_{j\in S_+,j\neq1}p_j
        -c\sum_{j\in S_-}p_j\right\}_d+\sum_{j\in S_+,j\neq 0}
      \{-p_jc\}_d+\sum_{j\in S_-}\{p_jc\}_d\right).
  $$
  Then from the assumption that $\sum_{j\in S_+}p_j=\sum_{j\in
    S_-}p_j$, we see that the expression can be simplified to
  $$
  \sum_{j\in T}a_j=\frac1d\left(\sum_{j\in S_+}\{-p_jc\}_d
    +\sum_{j\in S_-}\{p_jc\}_d\right)=w(c/d).
  $$
  Therefore,
  $$
  g(x;T)=\frac{\sum_{c=0}^{d-1}x^{w(c/d)}}{(1-x)^{|T|}}.
  $$
  This proves the lemma for the case $0,n\in T$.

  Finally, assume that $n\in T$ but $0\notin T$. As before, we find
  that for each integer $c$ with $0\le c\le d-1$, there is exactly a
  primitive element $u=\sum_{j\in T}a_j\e_j$ with $a_{n+1}=c/d$ and
  $$
  da_j=\begin{cases}
    \{-p_jc\}_d, &\text{if }j\in T_+, \\
    \{p_jc\}_d, &\text{if }j\in T_-.
  \end{cases}
  $$
  Then the weight of $u$ is equal to
  \begin{equation*}
    \begin{split}
      \sum_{j\in T}a_j
      &=\frac1d\left(\sum_{j\in T_+}\{-p_jc\}_d
        +\sum_{j\in T_-}\{p_jc\}_d\right) \\
      &=\frac1d\left(\sum_{j\in S_+}\{-p_jc\}_d
        +\sum_{j\in S_-}\{p_jc\}_d\right)=w(c/d).
    \end{split}
  \end{equation*}
  Again, we have $g(x;T)=f(x;T)/(1-x)^{|T|}$ with $f(x;T)$ given by
  \eqref{eq: f(x;T)}. This completes the proof of the lemma.
\end{proof}

\begin{cor}
  For a proper subset $T$ of $S$, let $f(x;T)$ be defined by
  \eqref{eq: f(x;T)}. Then the generating function
  $g(x)=\sum_{k=0}^\infty m(k)x^k$ for the number $m(k)$ of points of
  weight $k$ in $\Z^n\cap C(\Delta)$ is equal to $f(x)/(1-x)^n$, where
  \begin{equation} \label{eq: Hodge polynomial}
    \begin{split}
      f(x)&=\sum_{k=0}^{r-1}(-1)^{r-1-k}
      \sum_{T\subset S_+,|T|=k}f(x;T\cup S_-)(1-x)^{r-1-k} \\
      &=\sum_{k=0}^{s-1}(-1)^{s-1-k}
      \sum_{T\subset S_-,|T|=k}f(x;S_+\cup T)(1-x)^{s-1-k}.
    \end{split}
  \end{equation}
  Consequently, the Hodge number multiplicity $H_\Delta(k)$ is equal
  to the coefficient of $x^k$ in $f(x)$.
\end{cor}

\begin{example}
  Let $\G=[s,-1,\cdots,-1]$ with $r=1$, $p_1=s$, and $q_1=\cdots=q_{s}=1$, so $n=s$,
  $S_+=\{0\}$ and $S_-=\{1,\ldots,s\}$. This corresponds to the Dwork family. When $s=3$, as in Example \ref{eg:1} and when $s=4$, $f_{HD}(\l)=0$ is generically a K3 surface. According to the corollary,
  we have
  $$
  \sum_{k=0}^\infty m(k)x^k=\frac{f(x;S_-)}{(1-x)^s}.
  $$
  It is easy to see that primitive element of $C(\Delta_{S_-})$ are
  $$
  \frac cs(\e_1+\cdots+\e_{s})
  $$
  for $c=0,\ldots,s-1$. Thus,
  $$
  f(x;S_-)=1+x+\cdots+x^{s-1}
  $$
  Thus, $H_\Delta(k)=1$ for $k=0,\ldots,s-1$ and $H_\Delta(k)=0$ for
  $k\ge s$.
\end{example}

\begin{example}
  Let $\G=[6,-3,-2,-1]$ with $r=1$, $s=3$, and $n=3$. We
  have
  $$
  \sum_{k=0}^\infty m(k)x^k=\frac{f(x;S_-)}{(1-x)^3}.
  $$
  The primitive elements of $C(\Delta_{S_-})$ are
  $$
  \frac16\left(\{3c\}_6\e_1+\{2c\}_6\e_2+c\e_3\right),
  \qquad c=0,\ldots,5,
  $$
  whose weights are
  $$
  \begin{cases}
    0, &\text{if }c=0, \\
    1, &\text{if }c=1,2,3,4, \\
    2, &\text{if }c=5. \end{cases}
  $$
  Thus,
  $$
  \sum_{k=0}^\infty m(k)x^k=\frac{1+4x+x^2}{(1-x)^3}.
  $$
\end{example}

\begin{example}
  Let  $\G=[5,2,-6,-1]$ with $r=s=2$ and $n=3$. Then the polynomial
  $f(x)$ in $\sum_{k=0}^\infty m(k)x^k=f(x)/(1-x)^3$ is
  $$
  f(x)=f(x;\{1,2,3\})+f(x;\{1,2,4\})-f(x;\{1,2\})(1-x).
  $$
  By Lemma \ref{lemma: f(x;T)}, $f(x;\{1,2,3\})=f(x;\{1,2\})=1$,
  and
  $$
  f(x;\{1,2,4\})=\sum_{c=0}^5x^{w(c;\{1,2,4\})}=1+3x+2x^2.
  $$
  Thus,
  $$
  f(x)=1+(1+3x+2x^2)-(1-x)=1+4x+2x^2.
  $$
\end{example}

We now deduce an alternative expression for $f(x)$, from which we will
see that it is equal to the Hodge polynomial from the zigzag
algorithm. Before that, let us remark that since $w(t)$ is periodic of
period $1$, the polynomial $f(x;T)$ in \eqref{eq: f(x;T)} can also be
written as
$$
f(x;T)=\sum_{c=1}^dx^{w(c/d)}.
$$

We let $\sT$ denote the set
$$
  \sT:=\{t\in\Q:0<t\le 1, ~p_jt\in\Z\text{ for some }j\in S_+\}.
$$
For $t\in\sT$, we let
$$
T(t):=\{j\in S_+: p_jt\in\Z\}, \qquad
m(t):=|T(t)|.
$$
That is,
$$
\sT=\bigcup_{j\in S_+}\left\{\frac k{p_j}: k\in\Z, 1\le k\le p_j
  \right\},
$$
$T(t)$ is the set of indices $j$ such that $t$ is contained in
the set $\{k/p_j:k\in\Z,1\le k\le p_j\}$. 

Recall from \eqref{eq: Hodge polynomial} that
$$
  f(x)=\sum_{k=0}^{r-1}(-1)^{r-1-k}
    \sum_{T\subset S_+,|T|=k}f(x;T\cup S_-)(1-x)^{r-1-k},
$$
which is equal to
$$
\sum_{c/d\in\sT,(c,d)=1}x^{w(c/d)}
\sum_{T\subsetneq S_+,d|d_{T\cup S_-}}(x-1)^{r-1-|T|}.
$$
Now for $T\subset S_+$, we have $d|d_{T\cup S_-}$ if and only if
$T':=S\setminus(T\cup S_-)$ satisfies $T'\subseteq T(c/d)$. Thus,
\begin{equation} \label{eq: Hodge polynomial 1}
  \begin{split}
    f(x)
    &=\sum_{t\in\sT}x^{w(t)}\sum_{T'\subseteq T(t),T'\neq\emptyset}
    (x-1)^{|T'|-1} \\
    &=\sum_{t\in\sT}x^{w(t)}\sum_{j=1}^{m(t)}
    \binom{m(t)}j(x-1)^j \\
    &=\sum_{t\in\sT}x^{w(t)}(1+x+\cdots+x^{m(t)-1}).
  \end{split}
\end{equation}
This means that for each fraction $t$ in $\{k/p_j: 1\le
k\le p_j\}$, $j\in S_+$, if we let $m(t)$ denote the number of times
it appears in the sets, then $t$ contributes $m(t)$ terms continually
to $f(x)$, starting from the exponent $w(t)$ and ending at the
exponent $w(t)+m(t)-1$. Now we have
\begin{equation} \label{eq: w(t) temp}
  \begin{split}
    w(t)&=\sum_{j\in S_+}(-p_jt-\lfloor -p_jt\rfloor)
    +\sum_{j\in S_-}(p_jt-\lfloor p_jt\rfloor) \\
    &=-\sum_{j\in S_+}\lfloor-p_jt\rfloor
  -\sum_{j\in S_-}\lfloor p_jt\rfloor.
  \end{split}
\end{equation}
Using
$$
  \lfloor-p_jt\rfloor
  =-\#\left\{k\in\Z: 0<\frac k{p_j}<t\right\}-1
$$
and
$$
  \lfloor p_jt\rfloor
 =\#\left\{k\in\Z: 0<\frac k{p_j}\le t\right\},
$$
we find that
\begin{equation} \label{eq: w(t) 1}
w(t)=\sum_{j\in S_+}\#\left\{k\in\Z:0<\frac k{p_j}<t\right\}
-\sum_{j\in S_-}\#\left\{k\in\Z:0<\frac k{p_j}\le t\right\}+r.
\end{equation}
From this, we see that $f(x)$ is precisely equal to the polynomial
from the zigzag algorithm.

\section*{Acknowledgements}
Long is supported in part by the Simons Foundation grant  \#MP-TSM-00002492 and the LSU Micheal F. and Roberta Nesbit McDonald Professorship. She is grateful for the opportunity to visit Academia Sinica in Taiwan in 2023 during when this paper was written. Yang is supported by Grant 112-2115-M-002-002 of National Science and Technology Council of Taiwan (R.O.C.).

\end{document}